\documentclass[12pt]{amsart}
\usepackage{amsfonts,amssymb,amscd,amsmath,enumitem,verbatim}
\usepackage[latin1]{inputenc}
\usepackage{todonotes}
\usepackage{amscd}
\usepackage{latexsym}
\usepackage{color}
\usepackage{mathptmx}
\usepackage{multicol}
\input xy
\xyoption{all}
\usepackage{hyperref}
\hypersetup{colorlinks,linkcolor=blue ,citecolor=blue, urlcolor=blue}
\def\NZQ{\mathbb}               

\def\ZZ{{\NZQ Z}}
\def\RR{{\NZQ R}}

\def\frk{\mathfrak}               

\def\Phi{{\frk N}}
\def\xb{{\bold x}}

\def\opn#1#2{\def#1{\operatorname{#2}}} 
\opn\chara{char} 
\opn\length{\ell} 
\opn\pd{pd} 
\opn\rk{rk}
\opn\projdim{proj\,dim} 
\opn\injdim{inj\,dim} 
\opn\rank{rank}
\opn\depth{depth} 
\opn\grade{grade} 
\opn\height{height}
\opn\embdim{emb\,dim} 
\opn\codim{codim}

\opn\Tr{Tr} 
\opn\bigrank{big\,rank}
\opn\superheight{superheight}
\opn\lcm{lcm}
\opn\trdeg{tr\,deg}
\opn\reg{reg} 
\opn\lreg{lreg} 
\opn\ini{in} 
\opn\lpd{lpd}
\opn\size{size}
\opn\mult{mult}
\opn\dist{dist}
\opn\cone{cone}
\opn\lex{lex}
\opn\rev{rev}
\opn\div{div} \opn\Div{Div} \opn\cl{cl} \opn\Cl{Cl}
\opn\Spec{Spec} \opn\Supp{Supp} \opn\supp{supp} \opn\Sing{Sing}
\opn\Ass{Ass} \opn\Min{Min}
\opn\Ann{Ann} \opn\Rad{Rad} \opn\Soc{Soc}
\opn\Syz{Syz} \opn\Im{Im} \opn\Ker{Ker} \opn\Coker{Coker}
\opn\Am{Am} \opn\Hom{Hom} \opn\Tor{Tor} \opn\Ext{Ext}
\opn\End{End} \opn\Aut{Aut} \opn\id{id} \opn\ini{in}

\opn\nat{nat}
\opn\pff{pf}
\opn\Pf{Pf} \opn\GL{GL} \opn\SL{SL} \opn\mod{mod} \opn\ord{ord}
\opn\Gin{Gin}
\opn\Hilb{Hilb}\opn\adeg{adeg}\opn\std{std}\opn\ip{infpt}
\opn\Pol{Pol}
\opn\sat{sat}
\opn\Var{Var}
\opn\Gen{Gen}

\opn\aff{aff} \opn\con{conv} \opn\relint{relint} \opn\st{st}
\opn\lk{lk} \opn\cn{cn} \opn\core{core} \opn\vol{vol}
\opn\link{link} \opn\star{star}
\opn\gr{gr}

\def\Oc{{\mathcal O}}
\def\Pc{{\mathcal P}}

\def\Cc{{\mathcal C}}

\def\vol{{\textnormal{vol}}}
\def\conv{{\textnormal{conv}}}

\def\ord{{\textnormal{ord}}}

\def\pot#1#2{#1[\kern-0.28ex[#2]\kern-0.28ex]}

\opn\dirlim{\underrightarrow{\lim}}
\opn\inivlim{\underleftarrow{\lim}}
\def\Implies{\ifmmode\Longrightarrow \else
	\unskip${}\Longrightarrow{}$\ignorespaces\fi}
\def\implies{\ifmmode\Rightarrow \else
	\unskip${}\Rightarrow{}$\ignorespaces\fi}
\def\iff{\ifmmode\Longleftrightarrow \else
	\unskip${}\Longleftrightarrow{}$\ignorespaces\fi}

\let\:=\colon
\newtheorem{Theorem}{Theorem}[section]
\newtheorem{Lemma}[Theorem]{Lemma}
\newtheorem{Corollary}[Theorem]{Corollary}

\newtheorem{Example}[Theorem]{Example}

\newtheorem{Definition}[Theorem]{Definition}

\newtheorem{Question}[Theorem]{Question}
\numberwithin{equation}{section}

\let\epsilon\varepsilon
\let\phi=\varphi
\let\kappa=\varkappa
\opn\dis{dis}
\opn\height{height}
\opn\dist{dist}
\def\pnt{{\raise0.5mm\hbox{\large\bf.}}}

\opn\Lex{Lex}
\opn\conv{conv}

\newcommand\commentout[1]{ }

\begin{document}
	\title{Stanley's non-Ehrhart-positive order polytopes}
	
	\author{Fu Liu}
	\date{\today}
	\address{Department of Mathematics, University of California, One Shields Avenue, Davis, California 95616}
	\email{fuliu@math.ucdavis.edu}

	\author{Akiyoshi Tsuchiya}
	\address{Department of Pure and Applied Mathematics, Graduate School of Information Science and Technology, Osaka University, Suita, Osaka 565-0871, Japan}
	\email{a-tsuchiya@ist.osaka-u.ac.jp}
	
	\date{\today}
	\thanks{The first author was partially supported by partially supported by a grant from the Simons Foundation \#426756.
	The second author was partially supported by Grant-in-Aid for JSPS Fellows 16J01549.}
	
	\begin{abstract}
		We say a polytope is Ehrhart positive if all the coefficients in its Ehrhart polynomial are positive. Answering an Ehrhart positivity question posed on Mathoverflow, Stanley provided an example of a non-Ehrhart-positive order polytope of dimension $21$. 
Stanley's example comes from a certain family of order polytopes. 
In this paper, we study the Ehrhart positivity question on this family of polytopes.
By giving explicit formulas for the coefficients of the Ehrhart polynomials of these polytopes in terms of Bernolli numbers, we determine the sign of each Ehrhart coefficient of each polytope in the family.

As a consequence of our result, we conclude that for any positive integer $d \ge 21,$ there exists an order polytope of dimension $d$ that is not Ehrhart positive, and for any positive integer $\ell$, there exists an order polytope whose Ehrhart polynomial has precisely $\ell$ negative coefficients, which answers a question posed by Hibi.
We finish this article by discussing the existence of lower-dimensional order polytopes whose Ehrhart polynomials have a negative coefficient.
	\end{abstract}
	
	\maketitle	
	\section{Introduction}
	Let $\Pc \subset \RR^e$ be an {\em integral} convex polytope, that is, a convex polytope whose vertices have integer coordinates, of dimension $d$ (where $d \le e$). 
	Given integers $t=1,2,\ldots,$ we write $i(\Pc,t)$ for the number of integer points in $t\Pc$, where $t\Pc:=\{t \xb : \xb \in \Pc\}$ is the \emph{$t$th dilation} of $\Pc$.
	In other words, 
	\[
	i(\Pc,t)=|t\Pc \cap \ZZ^d|, \ \ t=1,2,\ldots.
	\]
	In the 1960's Ehrhart \cite{E62} discovered that $i(\Pc,t)$ is a polynomial in $t$ of degree $d.$
	Therefore, we call $i(\Pc,t)$ the {\em Ehrhart polynomial} of $\Pc$,
	and call the coefficients of $i(\Pc,t)$ the {\em Ehrhart coefficients} of $\Pc$. 

	It is well-known (we refer to the book \cite{CCD}) that the leading coefficient of $i(\Pc,t)$ equals the Euclidean volume of $\Pc$, the second highest coefficient equals half of the boundary volume of $\Pc$, and the constant term is always $1$.
	Thus, these three Ehrhart coefficients are always positive. We call the remaining Ehrhart coefficients, that is, the coefficients of $t^{d-2}, t^{d-3}, \dots, t^1$ of $i(\Pc,t)$, the \emph{middle Ehrhart coefficients} of $\Pc.$ 
	Unfortunately, it is not true that every middle Ehrhart coefficient is positive. One sees that middle Ehrhart coefficients only appear when $d \ge 3,$ and the first counterexample already comes up at dimension $3$ known as the Reeve's tetrahedron (see \cite[Example 3.22]{CCD}).
	Recently, it was shown in \cite{HHTY} that for each $d \ge 3,$ there exists an integral convex polytope of dimension $d$ such that all of its middle Ehrhart coefficients are simultaneously negative.

	We say that an integral convex polytope has {\em Ehrhart positivity} or is {\em Ehrhart positive} if all of its (middle) Ehrhart coefficients are positive. It is natural to ask the following question:
	\begin{Question}\label{ques:EP}
		Which families of integral convex polytopes have Ehrhart positivity?
	\end{Question}
	Quite a few families of polytopes, such as zonotopes \cite{zonotope}, the $y$-family of generalized permutohedra \cite{Post}, cyclic polytopes \cite{cyclic, high-integral} and flow polytopes arising from certain graphs \cite{MM}, are known to have this property.
	On the other hand, in addition to the aforementioned work \cite{HHTY} on examples of polytopes with negative Ehrhart coefficients for each dimension $d \ge 3,$ recent work has been done giving a negative answer to Question \ref{ques:EP} for certain families of polytopes by constructing examples in each sufficiently large dimension. For instance, as a consequence of work in \cite{smooth}, there exists a smooth polytope with a negative Ehrhart coefficient for each dimension $d \ge 3$ \cite[Theorem 1.2]{smooth}, and there exists a type-B generalized permutohedron that has a negative Ehrhart coefficient for each dimension $d \ge 7$ \cite[Proposition 1.3]{smooth}\cite[Remark 4.3.2]{survey}.
	Please refer to the survey paper \cite{survey} for a general discussion on Question \ref{ques:EP}. 
	
	In this paper, we will focus our attention to the family of order polytopes, which was first introduced and studied by Stanley \cite{StanleyOrderPoly}. We assume readers are familiar with the definition of a poset presented in \cite[Chapater 3]{StanleyEC1}. 
	\begin{Definition}[Definition 1.1 in \cite{StanleyOrderPoly}]
		The \emph{order polytope $\Oc_{P}$} of a finite poset $(P, \leq_{P})$
		is the subset of $\RR^{P} = \{f\colon P \rightarrow \RR \}$ defined
		by
		\begin{align*}
		&0 \leq f(i) \leq 1  &\qquad \text{for all } i\in P, \\
		&f(i)\leq f(j)      &\qquad \text{if } i\leq_{P} j.
		\end{align*}
	\end{Definition}

	\commentout{
\begin{Definition}
	The \emph{chain polytope $\Cc_P$} of a finite poset $(P, \leq_{P})$ is the subset of $\RR^{P} = \{g\colon P \rightarrow \RR \}$ defined by the conditions
	\begin{align*}
	&0 \leq g(i)   &\qquad \text{for all } i\in P, \\
	&g(i_1) + g(i_2)+ \dots g(i_k)\leq 1      &\qquad \text{for all chains } i_1 <_{P} i_2 <_{P} \dots <_{P} i_k \text{ of }P.\\
	\end{align*}
\end{Definition}
}

\commentout{
We know that the order polytope and the chain polytope of a $d$-element poset are $(0,1)$-polytopes of dimension $d$.
Order polytopes and chain polytopes have been studied in many different branches of mathematics including combinatorics (\cite{StanleyOrderPoly}), commutative algebra (\cite{Hibiring,HibiChainTriangulation}), algebraic geometry (\cite{hibivar}) and representation theory (\cite{hibirep}),
and they have good properties.
In particular, the order polytope and the chain polytope of a finite poset have a same Ehrhart polynomial (\cite[Theorem 4.1]{StanleyOrderPoly}).
However, they are not always Ehrhart positive.
}
We remark that Stanley also associates a ``chain polytope'' $\Cc_P$ to each finite poset \cite[Definition 2.1]{StanleyOrderPoly}. Since the order polytope $\Oc_P$ and the chain polytope $\Cc_P$ associated to a same poset have a same Ehrhart polynomial \cite[Theorem 4.1]{StanleyOrderPoly}, we omit the discussion on chain polytopes.

It is known that the order polytope of a $d$-element poset is a $(0,1)$-polytope of dimension $d.$ 
Order polytopes has been studied by many authors from different viewpoints of combinatorics \cite{StanleyOrderPoly}, commutative algebra \cite{Hibiring,HibiChainTriangulation}, algebraic geometry \cite{hibivar}, and representation theory \cite{hibirep},
and they have good properties (e.g., compressed polytopes, algebras with straightening rows on finite distributive lattices and Koszul algebras). 
However, order polytopes are not always Ehrhart positive.
Responding to an Ehrhart positivity question posed on mathoverflow, Stanley provided the following example \cite{StanleyAnswer}:
\begin{Example}
	For any nonnegative integer $k,$ let $Q_k$ be the poset with one minimal element covered by $k$ other elements.
	One can verify (see \cite[Exercise 3.164]{StanleyEC1}) that the Ehrhart polynomial of the order polytope $\Oc_{Q_k}$ is 
	\begin{equation}
	i(\Oc_{Q_k},t)=\sum_{i=1}^{t+1}i^k.
	\label{equ:ehrhart_Qk}
	\end{equation}
	Moreover, by direct computation, we see that the linear coefficient of $i(\Oc_{Q_{20}},t)$ equals $-168011/330,$ which is negative.
\end{Example}

Inspired by the above example and prior work on Question \ref{ques:EP}, we study the Ehrhart polynomial of Stanley's order polytope $\Oc_{Q_k}$ for any $k$, in search of counterexamples to Question \ref{ques:EP} for order polytopes of sufficiently large dimensions.
In our first main result below, we give an explicit formula for $i(\Oc_{Q_k},t)$ in terms of Bernoulli numbers. 

\begin{Theorem}
	\label{thm:formula}
	We have
	\[
		i(\Oc_{Q_k},t)=1+\sum_{j=1}^{k}\dfrac{(B_{k-j+1}+(k-j+1))}{k-j+1}\binom{k}{j}t^j+\dfrac{1}{k+1}t^{k+1},
	\] 
	where $B_n$ is the $n$th Bernoulli number.
\end{Theorem}
As a consequence of the above formula, we are able to determine the sign of each Ehrhart coefficient  of $\Oc_{Q_k}$ completely.
\begin{Theorem}
	\label{thm:sign}
	For any positive integer $j$ satisfying $1 \leq j \leq k-1$,
	the coefficient of $t^j$ of $i(\Oc_{Q_k},t)$ is negative if and only if $k-j+1 \geq 20$ and $4$ divides $k-j+1$.

	Hence, the order polytope $\Oc_{Q_k}$ is Ehrhart positive if and only if $k < 20$.
\end{Theorem}

We have the following immediate corollary to Theorem \ref{thm:sign}, providing an affirmative answer to a question asked by Hibi (private communication).
\begin{Corollary}
	\label{cor:hibi}
		For any positive integer $\ell$,
	there exists a finite poset $P$ such that $i(\Oc_{P},t)$ has precisely $\ell$ negative coefficients.	
\end{Corollary}

Another clear consequence of Theorem \ref{thm:sign} is that it settles our initial goal of finding order polytopes with a negative Ehrhart coefficient for every sufficiently large dimension. More precisely, since $\Oc_{Q_k}$ is of dimension $k+1,$ there exists a non-Ehrhart-positive order polytope for each dimension $d \ge 21.$ 
This naturally motivates us to consider the following question:

\begin{Question}
	\label{ques:lower}
Can we construct order polytopes of lower dimensions, i.e., $d \le 20$, that have a negative Ehrhart coefficient?
\end{Question}

We will discuss this question in the last part of this paper, and prove the following result:
\begin{Theorem}
	\label{thm:overall}
	For any positive integer $d \geq 14$, there exists a non-Ehrhart-positive order polytope of dimension $d$.

	Any order polytope of dimension $d \le 11$ is Ehrhart positive. 
\end{Theorem}

This article is organized as follows.
In Section \ref{sec:mainproof}, we recall definitions of Bernoulli polynomials and Bernoulli numbers, review their properties that are necessary for this paper, and then complete proofs for Theorems \ref{thm:formula} and \ref{thm:sign}.
In Section \ref{sec:lower}, we discuss Question \ref{ques:lower} on the existence of lower-dimensional non-Ehrhart-positive order polytopes, and provide construction of order polytopes with negative Ehrhart coefficients of dimension $d \in \{14, 15, \dots, 20\}.$

We finish this introduction with two questions for possible future research. 
First, given the results stated in Theorem \ref{thm:overall}, the only dimensions remain open in terms of Question \ref{ques:lower} are $12$ and $13$:
\begin{Question}
	\label{ques2}
	For $d=12$ or $13$, does there exist an order polytope of dimension $d$ such that its Ehrhart polynomial has a negative coefficient?
\end{Question}

Notice that it follows from the statement of Theorem \ref{thm:sign} that for each integer $k \ge 20,$ the coefficient of $t^j$ of $i(\Oc_{Q_k},t)$ is positive for all $j: k-18 \le j \le k+1$, and $t^{k-19}$ is the highest degree term that has a negative coefficient. Therefore, even though we were able to construct non-Ehrhart-positive order polytopes for each sufficiently large dimension, our construction does not give negative Ehrhart coefficients at higher degree terms. This is also true for the non-Ehrhart-positive order polytopes of dimension $d \in \{14,15,\dots, 20\}$ provided in Section \ref{sec:lower}; which all have negative coefficients appear in lower degree terms. Hence, it is natural to ask the following question:
\begin{Question}
Is it possible to construct order polytopes such that negative coefficients appear at higher degree terms (e.g., the third highest degree term,) in their Ehrhart polynomials, for each sufficiently large dimension $d$?
\end{Question}

\section{Proof of Theorems \ref{thm:formula} and \ref{thm:sign}}
\label{sec:mainproof}
	
	
		
		Recall the {\rm Bernoulli polynomials} $B_k(x)$ are defined by 
		\[
	\dfrac{te^{x t}}{e^t-1}=\sum_{k \geq 0}\dfrac{B_k(x)}{k!}t^k,
	\] 
		and the {\rm Bernoulli numbers} $B_k$ is then defined by $B_k:=B_k(1)$.
%
%
		We list the Bernoulli numbers $B_n$ for $n=0,1,2,\dots, 20$ in the table below.
		\begin{center}
		\bgroup
	\setlength\tabcolsep{3pt}
	\begin{tabular}{c|*{21}{c}}
		$n$ & 0 & 1 & 2 & 3 & 4 & 5 & 6 & 7 & 8 & 9 & 10 & 11 & 12 & 13 & 14 & 15 & 16 & 17 & 18 & 19 & 20 \\[1mm]
		\hline
		$B_n$ & $1$ & $\scriptstyle \frac{1}{2}$ & $\scriptstyle \frac{1}{6}$ & $0$ & $\scriptstyle -\frac{1}{30}$ & $0$ & $\scriptstyle \frac{1}{42}$ & $0$ & $\scriptstyle -\frac{1}{30}$ & $0$ & $\scriptstyle \frac{5}{66}$
		& $0$ & $\scriptstyle -\frac{691}{2730}$ & $0$ & $\scriptstyle \frac{7}{6}$ & $0$ & $\scriptstyle -\frac{3617}{510}$ & $0$ & $\scriptstyle \frac{43867}{798}$ & $0$ & $\scriptstyle -\frac{174611}{330}$
		\end{tabular}
		\egroup
	\end{center}
\commentout{
	\[
		B_0=1, B_1=\dfrac{1}{2}, B_2=\dfrac{1}{6}, B_4=-\dfrac{1}{30},
		B_6=\dfrac{1}{42},
		B_8=-\dfrac{1}{30},
		B_{10}=\dfrac{5}{66},
		B_{12}=-\dfrac{691}{2730},\]
		\[
		B_{14}=\dfrac{7}{6},
		B_{16}=-\dfrac{3617}{510},
		B_{18}=\dfrac{43867}{798},
		B_{20}=-\dfrac{174611}{330}.
		\]
	}

	Bernoulli numbers and Bernoulli polynomials are well-studied mathematical objects. Below we list a few of their properties (see the book \cite{Bernoulli}) that will be used in the proofs of Theorems \ref{thm:formula} and \ref{thm:sign}.
	\begin{enumerate}[label=(B\arabic*)]
		\item \label{itm:oddzero} 
	$B_k=0$ if $k$ is odd and $k \geq 3$.
\item \label{itm:evensign}  
	For any positive integer $k$ we have $B_{4k-2}>0$ and $B_{4k}<0$.
\item \label{itm:Bp_rec}
	$B_k(x+1)=B_k(x)+kx^{k-1}$.

	\item \label{itm:powersum} 	  
		For a positive integers $t$ and $n$, we have
	\[
	\sum_{i=1}^{t}i^n=\dfrac{1}{n+1}\sum_{i=0}^{n}\binom{n+1}{i+1} B_{n-i} t^{i+1} .
	\] 
	In particular, when $t=1,$ we get $\displaystyle \dfrac{1}{n+1}\sum_{i=0}^{n}\binom{n+1}{i+1} B_{n-i}  = 1.$
\item  \label{itm:evenrec}	
	For any positive integer $n\geq 2$, we have 
	\[
	B_{2n}=-\dfrac{1}{2n+1}\sum_{j=1}^{n-1}\binom{2n}{2j}B_{2j}B_{2(n-j)}.
	\]
		\end{enumerate}
	

We are now ready to prove Theorem \ref{thm:formula}.
\begin{proof}[Proof of Theorem \ref{thm:formula}]
	It follows from \eqref{equ:ehrhart_Qk} and Property \ref{itm:powersum} that 
	\begin{align*}
	i(\Oc_{Q_k},t)= \sum_{i=1}^{t+1}i^k &=\dfrac{1}{k+1}\sum_{i=0}^{k}(t+1)^{i+1} \binom{k+1}{i+1}B_{k-i}
	=\dfrac{1}{k+1}\sum_{i=0}^{k}\sum_{j=0}^{i+1} t^j \binom{i+1}{j} \binom{k+1}{i+1}B_{k-i} \\
	&=1+\dfrac{1}{k+1}\sum_{j=1}^{k+1}t^j\left(\sum_{i=j-1}^{k}\binom{i+1}{j}\binom{k+1}{i+1}B_{k-i}  \right),
\end{align*}
For any integers $k$ and $j$ satisfying $1 \leq j \leq k+1$, let $a_j^{(k)}$ be the coefficient of $t^j$ of $i(\Oc_{Q_k},t)$.
Then by the above formula, we have that 
\[
a_j^{(k)}=\dfrac{1}{k+1}\sum_{i=j-1}^{k}\binom{i+1}{j}\binom{k+1}{i+1}B_{k-i}.
\]
Using the facts that $B_0=1$, one sees that $a_{k+1}^{(k)}=\dfrac{1}{k+1}$.
Hence, it is left to show that for any integers $k$ and $j$ satisfying $1 \le j \le k,$
\begin{equation}
	a_j^{(k)} =\dfrac{(B_{k-j+1}+(k-j+1))}{k-j+1}\binom{k}{j}. 
	\label{equ:ajk}
\end{equation}

We first show \eqref{equ:ajk} holds when $j=1,$ equivalently, $a_1^{(k)}=B_k+k$ for any $k \geq 0$.
Considering the exponential generating function of $a_1^{(k)}:$
\[
\sum_{k=0}^{\infty}a_1^{(k)}\dfrac{t^k}{k!}=\sum_{k=0}^{\infty}
\left(
\sum_{i=0}^{k}\binom{k}{i}\cdot 1 \cdot B_{k-i}
\right)
\dfrac{t^k}{k!}\\
=\left(
\sum_{k=0}^{\infty}1 \cdot \dfrac{t^k}{k!}
\right)
\cdot
\left(
\sum_{k=0}^{\infty}B_i\dfrac{t^k}{k!}
\right)\\
=\dfrac{te^{2t}}{e^{t}-1}\\
=\sum_{k=0}^{\infty}B_k(2)\dfrac{t^k}{k!}.
\]
Therefore, $a_1^{(k)} = B_k(2) = B_k(1) +k = B_k+k,$ where the second equality follows from Property \ref{itm:Bp_rec} with $x=1$. 

Now one sees that it suffices to show that $a_{j+1}^{(k+1)}=\dfrac{k+1}{j+1}a_j^{(k)}$, for any $j \ge 1$, to finish the proof for \eqref{equ:ajk}.
Indeed,
\begin{align*}
a_{j+1}^{(k+1)}&=\dfrac{1}{k+2}\sum_{i=j}^{k+1}\binom{i+1}{j+1}\binom{k+2}{i+1}B_{k-i+1}
=\dfrac{1}{k+2}\sum_{i=j-1}^{k}\binom{i+2}{j+1}\binom{k+2}{i+2}B_{k-i}\\
&=\dfrac{k+1}{j+1}\left(
\dfrac{1}{k+1}\sum_{i=j-1}^{k}\binom{i+1}{j}\binom{k+1}{i+1}B_{k-i}
\right)
=\dfrac{k+1}{j+1}a_j^{(k)}.
\end{align*}
\commentout{
	Now, we compute $a_j^{(k)}$ for $2 \leq j \leq k-1$.
In fact, one has
\[
a_j^{(k)}=\dfrac{k}{j}a_{j-1}^{(k-1)}=\cdots=\dfrac{k}{j}\cdots\dfrac{k-j+2}{2}a_1^{(k-j+1)}=\dfrac{(B_{k-j+1}+(k-j+1))}{k-j+1}\binom{k}{j}.
\]
}
This completes the proof.
\end{proof}

Theorem \ref{thm:formula} states that the coefficient of $t^j$ of $i(\Oc_{Q_k}, t)$ is a positive multiple of $B_{k-j+1} + (k-j+1)$. Thus, one sees that Theorem \ref{thm:sign} is reduced to the following lemma.
\begin{Lemma}
 $B_k+k$ is negative if and only if $k \geq 20$ and  $4$ divides $k$ . 
\end{Lemma}
\begin{proof}
	By Properties \ref{itm:oddzero} and \ref{itm:evensign} and the values of $B_n$ listed at the beginning of the section, 
	one observes that $B_k + k$ is positive if $4$ does not divide $k$ or if $k < 20$.

	Now we assume that $k =4m$ with $m \geq 5$.
	We will show that $B_{4m} < -4m$ by induction on $m$. The base case when $m=5$ can be verified directly given we know that $B_{20} = -174611/330.$ Assume $m \ge 6$ and $B_{4m-4} <  -(4m-4)$.
	It follows from Property \ref{itm:evensign} that $B_{2j}B_{4m-2j} > 0$
	for positive integer $1 \leq j \leq 2m-1$.
	This together with Property \ref{itm:evenrec} gives:
	\begin{align*}
	B_{4m}&=-\dfrac{1}{4m+1}\sum_{j=1}^{2m-1}
	\left(
	\binom{4m}{2j}B_{2j}B_{4m-2j}
	\right)
	< -\dfrac{1}{4m+1}\binom{4m}{4}B_{4}B_{4m-4}
	\end{align*}
	It is easy to verify that for any $m \ge 5,$ we have
	\[ (4m-1)(2m-1) > 4m+1 \quad \text{and} \quad (4m-3)(m-1) > 90.\]
	Applying the fact that $B_4=-1/30$, the induction hypothesis, and the above two inequialities, we obtain
	\begin{align*}
	B_{4m}
	&<\dfrac{1}{30(4m+1)}\binom{4m}{4}B_{4m-4} <\dfrac{1}{30(4m+1)}\binom{4m}{4}(-4m+4)\\
	&= -\dfrac{2m (4m-1)(2m-1)(4m-3)(m-1)}{45(4m+1)} < -\dfrac{2m\cdot (4m+1) \cdot 90}{45(4m+1)} = -4m. \qedhere
	\end{align*}
	\commentout{We should show that 
	$g(x):=-64x^4+160x^3-140x^2+770x+174<0$ for $x \geq 5$.
	It is easy to see that for $i=1,\ldots,3$ and $x \geq 5$,
	one has $g^{(i)}(x) < 0$.
Hence we obtain $g(x) \leq g(5) <0$ for any $x \geq 5$.	}
\end{proof}

\section{Low-dimensional non-Ehrhart-positive order polytopes} \label{sec:lower}

In the same mathoverflow post \cite{StanleyAnswer} mentioned in the introduction,
Stanley stated that using Stembridge's posets package for Maple \cite{Stembridge}, one can check that the order polytope of any poset with at most eight elements is Ehrhart positive.
We extended his approach and verified that the order polytope of any poset with at most eleven elements is Ehrhart-positive by using Stembridge's posets package together with SageMath \cite{sage}.
We remark that the number of posets with eleven elements is $46749427$ and it took about three weeks to run the compute to check them.
It will be difficult to go beyond posets with more than $11$ elements because there are already $1104891746$ (more than one billion) posets with $12$ elements \cite{posetnumber}.

On the other hand, from Theorem \ref{thm:sign} 
we know that for any integer $d \geq 21$, there exists a non-Ehrhart-positive order polytope of dimension $d$. Therefore, in order to prove Theorem \ref{thm:overall}, it remains to show examples of non-Ehrhart-positive order polytopes of dimension $d = 14, 15, \dots, 20.$ The rest of this paper is devoted to constructing these examples, for which we need to recall two more concepts: ``$h^*$-polynomials'' of polytopes and ``ordinal sums'' of posets.

\commentout{
Before giving such examples, we recall the $h^*$-polynomial of integral convex polytopes.
Let $\Pc \subset \RR^d$ be an integral convex polytope of dimension $d$.
The generating function of the integer point enumerator, i.e., the formal power series
\[\text{Ehr}_\Pc(z)=1+\sum\limits_{t=1}^{\infty}i(\Pc,t)z^t\]
is called the \textit{Ehrhart series} of $\Pc$.
It is well known that it can be expressed as a rational function of the form
\[\text{Ehr}_\Pc(z)=\frac{h^*_0+h^*_1z+\cdots + h^*_dz^d}{(1-z)^{d+1}}.\]
 The polynomial in the numerator \[h^*(\Pc,z):=h^*_0+h^*_1z+\cdots+h^*_d z^d\]
is called
the \textit{$h^*$-polynomial} of $\Pc$. %
}
It is well-known that the generating function of the Ehrhart polynomial $i(\Pc,t)$ of an integral convex polytope $\Pc$ of dimension $d$ can be expressed as a rationl function of the form:
\[ 1+\sum\limits_{t=1}^{\infty}i(\Pc,t)z^t =\frac{h^*_0+h^*_1z+\cdots + h^*_dz^d}{(1-z)^{d+1}}.\]
The polynomial in the numerator  \[h^*(\Pc,z):=h^*_0+h^*_1z+\cdots+h^*_d z^d\] is called the \emph{$h^*$-polynomial} of $\Pc$.
We can extract the Ehrhart polynomial of $\Pc$ from its $h^*$-polynomial by the following formula:
\begin{equation}\label{equ:htoe}
i(\Pc,t)=\sum_{i=0}^{d}h^*_{i}\binom{t+d-i}{d}.
\end{equation}
It is a well-known result due to Stanley that $h^*(\Pc,z)$ is a polynomial in $z$ with non-negative integer coefficients of degree at most $d$. (See \cite{CCD} for more discussion on $h^*$-polynomials.)

\commentout{
Next, we recall an operation on posets.
 For finite posets $(P, \leq_P)$ and $(Q, \leq_Q)$ with $P \cap Q =\emptyset$,
the \textit{ordinal sum} of $P$ and $Q$ is the poset $(P \oplus Q,\leq_{P \oplus Q})$ 
on $P \oplus Q=P \cup Q$ such that $s \leq_{P \oplus Q} t$ }
The \emph{ordinal sum} of two disjoint finite posets $(P, \leq_P)$ and $(Q, \leq_Q)$ is the poset $(P \oplus Q,\leq_{P \oplus Q})$ on $P \oplus Q=P \cup Q$ such that $s \leq_{P \oplus Q} t$
if (a) $s,t \in P$ and $s \leq_{P} t$,
or (b) $s,t \in Q$ and $s \leq_{Q} t$,
or (c) $s \in P$ and $t \in Q$.
See Figure \ref{fig:ordinalsum} for an example.
	\begin{figure}[h]
		\center
\begin{picture}(300,60)(20,5)
\put(20,8){$P$}
\put(50,30){\circle*{5}}
\put(50,0){\circle*{5}}
\put(50,0){\line(0,1){30}}

\put(125,15){$Q$}
\put(150,0){\circle*{5}}
\put(110,0){\circle*{5}}

\put(215,55){$P \oplus Q$}
\put(250,60){\circle*{5}}
\put(210,60){\circle*{5}}
\put(230,30){\circle*{5}}
\put(230,0){\circle*{5}}
\put(230,0){\line(0,1){30}}
\put(230,30){\line(2,3){20}}
\put(230,30){\line(-2,3){20}}
\end{picture}
\caption{The ordinal sum of a $2$-elements chain and a $2$-elements antichain.}\label{fig:ordinalsum}
	\end{figure}
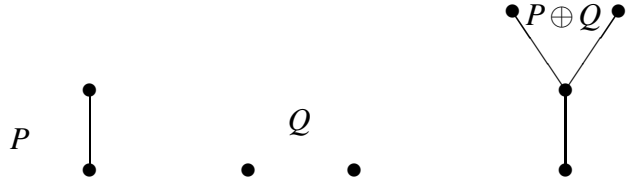

\begin{Lemma}[{\cite[Proposition 7.4]{HKT}}]
	\label{lem:ordinal}
	Suppose $(P, \leq_P)$ and $(Q, \leq_Q)$ be two disjoint finite posets. 
	Then 
	\[
	h^*(\Oc_{P \oplus Q}, z)=h^*(\Oc_P,z) \cdot h^*(\Oc_Q,z).
	\]
\end{Lemma}

Now we can construct examples of lower-dimensional non-Ehrhart-positive order polytopes, finishing the proof of Theorem \ref{thm:overall}.
\begin{Example}
	\label{ex:lower}
	For a positive integer $k$,
	let $P_k$ be a $k$-element antichain.
	Then the order polytope $\Oc_{P_k}$ is a $k$-dimensional unit cube.
	Thus, $h^*(\Oc_{P_k},z)=A_k(z)$, where $A_k(z)$ is the Eulerian polynomial of degree $k-1$ (see e.g., \cite[Example 4.6.15]{StanleyEC1}). 
	We list below the Eulerian polynomials $A_k(z)$ for $6 \leq k \leq 10$:
	\[
	A_6(z)=(1+z^5)+57(z+z^4)+302(z^2+z^3),
	\]
	\[
	A_7(z)=(1+z^6)+120(z+z^5)+1191(z^2+z^4)+2416z^3,
	\]
	\[
	A_8(z)=(1+z^7)+247(z+z^6)+4293(z^2+z^5)+15619(z^3+z^4),
	\]
		\[
	A_9(z)=(1+z^8)+502(z+z^7)+14608(z^2+z^6)+88234(z^3+z^5)+156190z^4,
	\]
		\[
	A_{10}(z)=(1+z^9)+1013(z+z^8)+47840(z^2+z^7)+455192(z^3+z^6)+1310354(z^4+z^5).
	\]
	For positive integers $m,n$, let $P_{m,n}$ be the ordinal sum of $P_m$ and $P_n$.
	Note that $P_{k,1}=Q_k$. So this is a naturally generalization of the family of posets $Q_k$ considered in Section \ref{sec:mainproof}.
	By Lemma \ref{lem:ordinal},
	one has 
	\[
	h^*(\Oc_{P_{m,n}},z)=h^*(\Oc_{P_m},z)\cdot h^*(\Oc_{P_{n}},z)=A_m(z)\cdot A_n(z).
	\]
	We then apply \eqref{equ:htoe} to compute Ehrhart polynomials of $\Oc_{P_{n,n}}$ for $6 \le n \le 10$, and Ehrhart polynomials of $\Oc_{P_{n,n+1}}$ for $6 \le n \le 9.$ The computation results are listed in Table \ref{tab:ehrhart}, from which we see that the first two order polytopes, $\Oc_{P_{6,6}}$ and $\Oc_{P_{6,7}}$, are Ehrhart-positive, and the rest of the order polytopes all have negative Ehrhart coefficients. Thus, $\Oc_{P_{7,7}},$ $\Oc_{P_{7,8}},$ $\dots,$ $\Oc_{P_{10,10}}$ are non-Ehrhart-positive order polytopes of dimension $14, 15, \dots, 20$, respectively. Unfortunately, the same construction does not give us examples of non-Ehrhart-positive order polytopes of dimension $12$ or $13.$
\commentout{	for $7 \leq n \leq 10$, 
	the Ehrhart polynomial $i(\Oc_{P_{n,n}},z)$ has a negative coefficient, and for  $7 \leq n \leq 9$,
		the Ehrhart polynomial $i(\Oc_{P_{n,n+1}},z)$ has a negative coefficient.
			However, $\Oc_{P_{6,6}}$ and $\Oc_{P_{6,7}}$ are Ehrhart-positive.}		
Moreover, we have verified that any order polytope $\Oc_P$ of dimension $12$ or $13$ is Ehrhart positive if $P$ can be expressed either as an ordinal sum of disjoint posets where each poset has at most $8$ elements, or as an ordinal sum of antichains (without the restriction on sizes).

\end{Example}

\begin{table}[t]
	\centering
	\caption{The Ehrhart polynomials of $\Oc_{P_{m,n}}$}
	\label{tab:ehrhart}
	\begin{tabular}{|l|l|}
		\hline
		$P$         & $i(\Oc_{P},t)$                                                                                                                                                                                                                                                                                                                                                                                                                                                                                                                                                                                                                                                           \\ \hline
		$P_{6,6}$   & \begin{tabular}[c]{@{}l@{}}$1+{\frac {75\,t}{22}}+{\frac {824\,{t}^{2}}{77}}+{\frac {160\,{t}^{3}}{7}}+{\frac {181\,{t}^{4}}{6}}+{\frac {765\,{t}^{5}}{28}}+{\frac {127\,{t}^{6}}{7}}+9\,{t}^{7}+{\frac {93\,{t}^{8}}{28}}+{\frac {25\,{t}^{9}}{28}}+\frac{1}{6}t^{10}+{\frac {3\,{t}^{11}}{154}}$\\$+{\frac {{t}^{12}}{924}}$\end{tabular}                                                                                                                                                                                                                                                                                                                                 \\ \hline
		$P_{6,7}$   & \begin{tabular}[c]{@{}l@{}}$1+{\frac {61751\,t}{15015}}+{\frac {555\,{t}^{2}}{44}}+{\frac {928\,{t}^{3}}{33}}+{\frac {83\,{t}^{4}}{2}}+{\frac {1273\,{t}^{5}}{30}}+{\frac {255\,{t}^{6}}{8}}+{\frac {127\,{t}^{7}}{7}}+ {\frac {63\,{t}^{8}}{8}}+{\frac {31\,{t}^{9}}{12}}+\frac{5}{8}{t}^{10}$\\$+{\frac {7\,{t}^{11}}{66}}+{\frac {{t}^{12}}{88}}+{\frac {{t}^{13}}{1716}}$\end{tabular}                                                                                                                                                                                                                                                                                     \\ \hline
		$P_{7,7}$   & \begin{tabular}[c]{@{}l@{}}$1-{\frac {3041\,t}{1430}}+{\frac {18397\,{t}^{2}}{4290}}+{\frac {1365\,{t}^{3}}{44}}+{\frac {602\,{t}^{4}}{11}}+{\frac {301\,{t}^{5}}{5}}+{\frac {8953\,{t}^{6}}{180}}+{\frac {255\,{t}^{7}}{8}}+{\frac {127\,{t}^{8}}{8}}+{\frac {49\,{t}^{9}}{8}}$\\$+{\frac {217\,{t}^{10}}{120}}+{\frac {35\,{t}^{11}}{88}}+{\frac {49\,{t}^{12}}{792}}+{\frac {7\,{t}^{13}}{1144}}+{\frac {{t}^{14}}{3432}}$\end{tabular}                                                                                                                                                                                                                              \\ \hline
		$P_{7,8}$   & \begin{tabular}[c]{@{}l@{}}$1-{\frac {1633\,t}{2145}}+{\frac {11261\,{t}^{2}}{2860}}+{\frac {208909\,{t}^{3}}{6435}}+{\frac {6125\,{t}^{4}}{88}}+{\frac {14441\,{t}^{5}}{165}}+{\frac {959\,{t}^{6}}{12}}+{\frac {5113\,{t}^{7}}{90}}+{\frac {255\,{t}^{8}}{8}}+{\frac {127\,{t}^{9}}{9}}$\\$+{\frac {49\,{t}^{10}}{10}}+{\frac {217\,{t}^{11}}{165}}+{\frac {35\,{t}^{12}}{132}}+{\frac {49\,{t}^{13}}{1287}}+{\frac {{t}^{14}}{286}}+{\frac {{t}^{15}}{6435}}$\end{tabular}                                                                                                                                                                                      \\ \hline
		$P_{8,8}$   & \begin{tabular}[c]{@{}l@{}}$1-{\frac {9905\,t}{286}}-{\frac {81704\,{t}^{2}}{2145}}+{\frac {18740\,{t}^{3}}{429}}+{\frac {137692\,{t}^{4}}{1287}}+{\frac {1358\,{t}^{5}}{11}}+{\frac {57736\,{t}^{6}}{495}}+{\frac {1364\,{t}^{7}}{15}}+{\frac {511\,{t}^{8}}{9}}+{\frac {85\,{t}^{9}}{3}}$\\$+{\frac {508\,{t}^{10}}{45}}+{\frac {196\,{t}^{11}}{55}}+{\frac {434\,{t}^{12}}{495}}+{\frac {70\,{t}^{13}}{429}}+{\frac {28\,{t}^{14}}{1287}}+{\frac {4\,{t}^{15}}{2145}}+{\frac {{t}^{16}}{12870}}$\end{tabular}                                                                                                                                                   \\ \hline
		$P_{8,9}$   & \begin{tabular}[c]{@{}l@{}}$1-{\frac {1063343\,t}{36465}}-{\frac {29713\,{t}^{2}}{572}}+{\frac {126224\,{t}^{3}}{6435}}+{\frac {17882\,{t}^{4}}{143}}+{\frac {75956\,{t}^{5}}{429}}+{\frac {1967\,{t}^{6}}{11}}+{\frac {24644\,{t}^{7}}{165}}+{\frac {1023\,{t}^{8}}{10}}$\\$+{\frac {511\,{t}^{9}}{9}}+{\frac {51\,{t}^{10}}{2}}+{\frac {508\,{t}^{11}}{55}}+{\frac {147\,{t}^{12}}{55}}+{\frac {434\,{t}^{13}}{715}}+{\frac {15\,{t}^{14}}{143}}+{\frac {28\,{t}^{15}}{2145}}+{\frac {3\,{t}^{16}}{2860}}+{\frac {{t}^{17}}{24310}}$\end{tabular}                                                                                                                \\ \hline
		$P_{9,9}$   & \begin{tabular}[c]{@{}l@{}}$1-{\frac {1285677\,t}{4862}}-{\frac {7364613\,{t}^{2}}{24310}}+{\frac{89157\,{t}^{3}}{572}}+{\frac {246946\,{t}^{4}}{715}}+{\frac {195381\,{t}^{5}}{715}}+{\frac {173242\,{t}^{6}}{715}}+{\frac {2439\,{t}^{7}}{11}}$\\$+{\frac {9204\,{t}^{8}}{55}}+{\frac {1023\,{t}^{9}}{10}}+{\frac {511\,{t}^{10}}{10}}+{\frac {459\,{t}^{11}}{22}}+{\frac {381\,{t}^{12}}{55}}+{\frac {1323\,{t}^{13}}{715}}+{\frac {279\,{t}^{14}}{715}}+{\frac {9\,{t}^{15}}{143}}+{\frac {21\,{t}^{16}}{2860}}$\\$+{\frac {27\,{t}^{17}}{48620}}+{\frac {{t}^{18}}{48620}}$\end{tabular}                                                                         \\ \hline
		$P_{9,10}$  & \begin{tabular}[c]{@{}l@{}}$1-{\frac {220154521\,t}{969969}}-{\frac {20069739\,{t}^{2}}{48620}}-{\frac {454951\,{t}^{3}}{12155}}+{\frac {453525\,{t}^{4}}{1144}}+{\frac {64414\,{t}^{5}}{143}}+{\frac {548577\,{t}^{6}}{1430}}+{\frac {1718664\,{t}^{7}}{5005}}$\\$+{\frac {3060\,{t}^{8}}{11}}+{\frac {12277\,{t}^{9}}{66}}+{\frac {1023\,{t}^{10}}{10}}+{\frac {511\,{t}^{11}}{11}}+{\frac {765\,{t}^{12}}{44}}+{\frac {762\,{t}^{13}}{143}}+{\frac {189\,{t}^{14}}{143}}+{\frac {186\,{t}^{15}}{715}}+{\frac {45\,{t}^{16}}{1144}}$\\$+{\frac {21\,{t}^{17}}{4862}}+{\frac {3\,{t}^{18}}{9724}}+{\frac {{t}^{19}}{92378}}$\end{tabular}                             \\ \hline
		$P_{10,10}$ & \begin{tabular}[c]{@{}l@{}}$1-{\frac {135276175\,t}{58786}}-{\frac {2250043660\,{t}^{2}}{969969}}+{\frac {4024062\,{t}^{3}}{2431}}+{\frac {11364453\,{t}^{4}}{4862}}+{\frac {461265\,{t}^{5}}{572}}+{\frac {150248\,{t}^{6}}{429}}$\\ $+{\frac {445884\,{t}^{7}}{1001}}+{\frac {426227\,{t}^{8}}{1001}}+{\frac {6825\,{t}^{9}}{22}}+{\frac {2047\,{t}^{10}}{11}}+93\,{t}^{11}+{\frac {2555\,{t}^{12}}{66}}+{\frac {3825\,{t}^{13}}{286}}+{\frac {3810\,{t}^{14}}{1001}}$\\$+{\frac {126\,{t}^{15}}{143}}+{\frac {93\,{t}^{16}}{572}}+{\frac {225\,{t}^{17}}{9724}}+{\frac {35\,{t}^{18}}{14586}}+{\frac {15\,{t}^{19}}{92378}}+{\frac {{t}^{20}}{184756}}$\end{tabular} \\ \hline
	\end{tabular}
\end{table}

\end{document}